\documentclass{amsart}
\usepackage{amscd,amsmath,amssymb}
\usepackage{pstricks}
\usepackage{latexsym,amsbsy,mathrsfs}
\usepackage{xy}
\usepackage{xypic}
\usepackage{pgf,tikz}

\newtheorem{thm}{Theorem}[section]
\newtheorem{lem}[thm]{Lemma}

\theoremstyle{definition}

\newtheorem{defn}[thm]{Definition}

\newtheorem{rem}[thm]{Remark}

\numberwithin{equation}{thm}

\begin{document}
\title[ Homotopy cartesian diagrams in $n$-angulated categories ]
{Homotopy cartesian diagrams in $n$-angulated categories }

\author{Zengqiang Lin and Yan Zheng}
\address{ School of Mathematical sciences, Huaqiao University,
Quanzhou\quad 362021,  China.} \email{zqlin@hqu.edu.cn}
\address{ School of Mathematical sciences, Huaqiao University,
Quanzhou\quad 362021,  China.} \email{zy112291@163.com}

\thanks{This work was supported  by  the Science Foundation of Huaqiao University (Grants No. 2014KJTD14)}



\begin{abstract} 
It has been proved by Bergh and Thaule that  the higher mapping cone axiom is equivalent to the higher octahedral axiom for $n$-angulated categories. In this note, we use homotopy cartesian diagrams to give several new equivalent statements of the higher mapping cone axiom, which are applied to explain the higher octahedral axiom.
\end{abstract}

\maketitle

\section{introduction}

Let $n$ be an integer greater than or equal to three. The notion of $n$-angulated category is introduced by Geiss, Keller and Oppermann in \cite{[GKO]} as the axiomatization of some $(n-2)$-cluster tilting subcategories of  triangulated categories. In particular, a 3-angulated category is the classical triangulated category. The examples of $n$-angulated categories can be found in \cite{[GKO],[BJT]}. Bergh and Thaule discussed the axioms of $n$-angulated categories  systematically \cite{[BT1]}. They showed that for $n$-angulated categories, the higher mapping cone axiom is equivalent to the higher octahedral axiom. It is proved in \cite{[ABT]} that the morphism axiom is redundant for $n$-angulated categories.

The first aim and motivation of this note is to understand the higher octahedral axiom. The $n$-angle induced by the higher octahedral axiom seems very mysterious, because it involves a lot of objects and morphisms. How do they behave together? 
What are the morphisms of $n$-angles we can infer from the higher octahedral axiom?

 The second motivation is to discuss other equivalent statements of the higher mapping cone axiom. As we all know, there are quite a few equivalent statements of octahedral axiom such as homotopy cartesian axiom, cobase change axiom and so on, which are used to construct triangles under varied conditions. Can we get their higher versions?

It will turn out that homotopy cartesian diagrams provide a useful method to achieve our two goals. Actually, the notion of homotopy cartesian diagrams in $n$-angulated categories is first introduced in \cite{[BT1]}. For triangulated case, we can find homotopy cartesian square in \cite{[PS],[Ne],[Kr]}. Since a homotopy cartesian square is the triangulated analogue of pushout and pullback square in an abelian category, inspired by the definition of $n$-pushout and $n$-pullback diagrams in $n$-abelian categories \cite{[J]}, 
we give an equivalent definition of homotopy cartesian diagrams to avoid dealing with the symbols $(-1)^i$ in the $n$-angle (see Remark \ref{2.1} for details). Then we prove that the higher mapping cone axiom is equivalent to the higher homotopy cartesian axiom, which is implied but not state explicitly in \cite{[BT1]}. At last, we will use homotopy cartesian diagrams to give several other somewhat new equivalent statements of the higher mapping cone axiom.

This note is organized as follows. In Section 2, we first recall the definition of $n$-angulated categories,  then introduce the notion of homotopy cartesian diagrams  and provide some needed facts. In Section 3, we use homotopy cartesian diagrams to give several new equivalent statements of the higher mapping cone axiom.

\section{$n$-angulated categories and homotopy cartesian diagrams }
For convenience, we recall the definition of $n$-angulated categories from \cite{[GKO]}.
Let $\mathcal{C}$ be an additive category equipped with an automorphism $\Sigma:\mathcal{C}\rightarrow\mathcal{C}$. 
 An $n$-$\Sigma$-$sequence$ in $\mathcal{C}$ is a sequence of morphisms
$$X_\bullet= (X_1\xrightarrow{f_1}X_2\xrightarrow{f_2}X_3\xrightarrow{f_3}\cdots\xrightarrow{f_{n-1}}X_n\xrightarrow{f_n}\Sigma X_1).$$
Its $left\ rotation$ is the $n$-$\Sigma$-sequence
$$X_2\xrightarrow{f_2}X_3\xrightarrow{f_3}X_4\xrightarrow{f_4}\cdots\xrightarrow{f_{n-1}}X_n\xrightarrow{f_n}\Sigma X_1\xrightarrow{(-1)^n\Sigma f_1}\Sigma X_2.$$  
An $n$-$\Sigma$-sequence $X_\bullet$ is $exact$ if the induced sequence
$$\cdots\rightarrow \mbox{Hom}_\mathcal{C}(-,X_1)\rightarrow \mbox{Hom}_\mathcal{C}(-,X_2)\rightarrow\cdots\rightarrow \mbox{Hom}_\mathcal{C}(-,X_n)\rightarrow \mbox{Hom}_\mathcal{C}(-,\Sigma X_1)\rightarrow\cdots$$
is exact. 
A $morphism$ of $n$-$\Sigma$-sequences is a sequence of morphisms $\varphi_\bullet=(\varphi_1,\varphi_2,\cdots,\varphi_n)$ such that the following diagram
$$\xymatrix{
X_1 \ar[r]^{f_1}\ar[d]^{\varphi_1} & X_2 \ar[r]^{f_2}\ar[d]^{\varphi_2} & X_3 \ar[r]^{f_3}\ar[d]^{\varphi_3} & \cdots \ar[r]^{f_{n-1}}& X_n \ar[r]^{f_n}\ar[d]^{\varphi_n} & \Sigma X_1 \ar[d]^{\Sigma \varphi_1}\\
Y_1 \ar[r]^{g_1} & Y_2 \ar[r]^{g_2} & Y_3 \ar[r]^{g_3} & \cdots \ar[r]^{g_{n-1}} & Y_n \ar[r]^{g_n}& \Sigma Y_1\\
}$$
commutes where each row is an $n$-$\Sigma$-sequence. It is an {\em isomorphism} if $\varphi_1, \varphi_2, \cdots, \varphi_n$ are all isomorphisms in $\mathcal{C}$.

\begin{defn} (\cite{[GKO]}) Let $\mathcal{C}$ be an additive category, $\Sigma$ an automorphism of $\mathcal{C}$ and $\Theta$ a collection of $n$-$\Sigma$-sequences. We call $(\mathcal{C},\Sigma,\Theta)$ a $pre$-$n$-$angulated\ category$ and call the elements of $\Theta$ $n$-$angles$ if $\Theta$ satisfies the following axioms:

(N1) (a) $\Theta$ is closed under isomorphisms, direct sums and direct summands.

(b) For each object $X\in\mathcal{C}$ the trivial sequence
$$X\xrightarrow{1}X\rightarrow 0\rightarrow\cdots\rightarrow 0\rightarrow \Sigma X$$
belongs to $\Theta$.

(c) For each morphism $f_1:X_1\rightarrow X_2$ in $\mathcal{C}$, there exists an $n$-$\Sigma$-sequence in $\Theta$ whose first morphism is $f_1$.

(N2) An $n$-$\Sigma$-sequence belongs to $\Theta$ if and only if its left rotation belongs to $\Theta$.

(N3) Each commutative diagram
$$\xymatrix{
X_1 \ar[r]^{f_1}\ar[d]^{\varphi_1} & X_2 \ar[r]^{f_2}\ar[d]^{\varphi_2} & X_3 \ar[r]^{f_3}\ar@{-->}[d]^{\varphi_3} & \cdots \ar[r]^{f_{n-1}}& X_n \ar[r]^{f_n}\ar@{-->}[d]^{\varphi_n} & \Sigma X_1 \ar[d]^{\Sigma \varphi_1}\\
Y_1 \ar[r]^{g_1} & Y_2 \ar[r]^{g_2} & Y_3 \ar[r]^{g_3} & \cdots \ar[r]^{g_{n-1}} & Y_n \ar[r]^{g_n}& \Sigma Y_1\\
}$$ with rows in $\Theta$ can be completed to a morphism of  $n$-$\Sigma$-sequences.

 If $\Theta$ moreover satisfies the following axiom, then  $(\mathcal{C},\Sigma,\Theta)$ is called an $n$-$angulated$ $category$:

(N4) In the situation of (N3), the morphisms $\varphi_3, \varphi_4, \cdots,\varphi_n$ can be chosen such that the mapping cone
$$X_2\oplus Y_1\xrightarrow{\left(
                              \begin{smallmatrix}
                                -f_2 & 0 \\
                                \varphi_2 & g_1 \\
                              \end{smallmatrix}
                            \right)}
 X_3\oplus Y_2 \xrightarrow{\left(
                              \begin{smallmatrix}
                                -f_3 & 0 \\
                                \varphi_3 & g_2 \\
                              \end{smallmatrix}
                            \right)}
 \cdots \xrightarrow{\left(
                            \begin{smallmatrix}
                               -f_n & 0 \\
                                \varphi_n & g_{n-1} \\
                             \end{smallmatrix}
                           \right)}
 \Sigma X_1\oplus Y_n \xrightarrow{\left(
                              \begin{smallmatrix}
                                -\Sigma f_1 & 0 \\
                                \Sigma\varphi_1 & g_n \\
                              \end{smallmatrix}
                            \right)}
 \Sigma X_2\oplus \Sigma Y_1 \\
$$
belongs to $\Theta$.
\end{defn}

The following theorem means that for $n$-angulated categories the higher mapping cone axiom is equivalent to the higher octahedral axiom.

\begin{thm} (\cite[Theorem 4.4]{[BT1]})
 Let $(\mathcal{C},\Sigma,\Theta)$ be a pre-$n$-angulated category. Then $\Theta$ satisfies (N4) if and only if $\Theta$ satisfies (N4$^*$):

Given the following commutative diagram
 $$\xymatrix{
X_1 \ar[r]^{f_1}\ar@{=}[d] & X_2 \ar[r]^{f_2}\ar[d]_{\varphi_2} & X_3 \ar[r]^{f_3}\ar@{-->}[d]_{\varphi_3} & X_4 \ar[r]^{f_4}\ar@{-->}[ldd]^{\phi_4}\ar@{-->}[d]^{\varphi_4} & \cdots \ar[r]^{f_{n-2}} & X_{n-1} \ar[r]^{f_{n-1}} \ar@{-->}[d]_{\varphi_{n-1}} & X_n \ar[r]^{f_n}\ar@{-->}[d]^{\varphi_n}\ar@{-->}[ldd]^{\phi_n} & \Sigma X_1 \ar@{=}[d]\\
X_1 \ar[r]^{g_1}\ar[d]^{f_1} & Y_2 \ar[r]^{g_2}\ar@{=}[d] & Y_3 \ar[r]^{g_3\ \ }\ar@{-->}[d]_{\psi_3} &Y_4 \ar[r]^{g_4}\ar@{-->}[d]^{\psi_4} & \cdots\ar[r]^{g_{n-2}} & Y_{n-1} \ar[r]^{g_{n-1}\ \ \ }\ar@{-->}[d]_{\psi_{n-1}} & Y_n \ar[r]^{g_n}\ar@{-->}[d]^{\psi_n}& \Sigma X_1\ar[d]^{\Sigma f_1}\\
X_2 \ar[r]^{\varphi_2} & Y_2 \ar[r]^{h_2} & Z_3 \ar[r]^{h_3} &Z_4 \ar[r]^{h_4} & \cdots \ar[r]^{h_{n-2}} & Z_{n-1} \ar[r]^{h_{n-1}} & Z_n \ar[r]^{h_n}& \Sigma X_2\\
}$$ whose rows are $n$-angles, there exist morphisms $\varphi_i: X_i\rightarrow Y_i$ for $3\leq i\leq n$, $\psi_j:Y_j\rightarrow Z_j$ for $3\leq j\leq n$ and $\phi_k:X_k\rightarrow Z_{k-1}$ for $4\leq k\leq n$ such that each square in the above diagram commutes and the following $n$-$\Sigma$-sequence
$$
 X_3\xrightarrow{\left(
             \begin{smallmatrix}
             f_3 \\
             \varphi_3\\
             \end{smallmatrix}
           \right)} X_4\oplus Y_3 \xrightarrow{\left(
             \begin{smallmatrix}
             -f_4 & 0 \\
             \varphi_4 & -g_3\\
              \phi_4 & \psi_3\\
             \end{smallmatrix}
           \right)}  X_5\oplus Y_4\oplus Z_3\xrightarrow{\left(
             \begin{smallmatrix}
             -f_5 & 0 & 0 \\
             -\varphi_5 & -g_4 & 0 \\
             \phi_5 & \psi_4 & h_3\\
             \end{smallmatrix}
           \right)} X_6\oplus Y_5\oplus Z_4$$
 $$ \xrightarrow{\left(
             \begin{smallmatrix}
             -f_6 & 0 & 0 \\
             \varphi_6 & -g_5 & 0 \\
             \phi_6 & \psi_5 & h_4\\
             \end{smallmatrix}
           \right)}  \cdots  \xrightarrow{\left(
             \begin{smallmatrix}
             -f_{n-1} & 0 & 0 \\
             (-1)^{n-1}\varphi_{n-1} & -g_{n-2} & 0 \\
             \phi_{n-1} & \psi_{n-2} & h_{n-3}\\
             \end{smallmatrix}
           \right)} X_n\oplus Y_{n-1}\oplus Z_{n-2}$$
 $$\begin{gathered}\xrightarrow{\left(
             \begin{smallmatrix}
             (-1)^n\varphi_{n} & -g_{n-1} & 0 \\
             \phi_{n} & \psi_{n-1} & h_{n-2}\\
             \end{smallmatrix}
           \right)}  Y_n\oplus Z_{n-1}\xrightarrow{\left(
             \begin{smallmatrix}
             \psi_n & h_{n-1} \\
             \end{smallmatrix}
           \right)} Z_n \xrightarrow{\Sigma f_2\cdot h_n}  \Sigma X_3 \ \ \ \ \ \ \ \ \ \ \ \ \ \ \ \ (2.1)
 \end{gathered}$$ belongs to $\Theta$.
\end{thm}

In the rest of this section, we assume that all $n$-angles are in a pre-$n$-angulated category $(\mathcal{C},\Sigma,\Theta)$.

\begin{defn} (cf. \cite{[BT1]})
 The following commutative diagram
$$\xymatrix{
 X_1 \ar[r]^{f_1}\ar[d]^{\varphi_1} & X_2 \ar[r]^{f_2}\ar[d]^{\varphi_2} & \cdots \ar[r]^{f_{n-3}}& X_{n-2}\ar[r]^{f_{n-2}}\ar[d]^{\varphi_{n-2}}& X_{n-1} \ar[d]^{\varphi_{n-1}} \\
 Y_1 \ar[r]^{g_1} & Y_2 \ar[r]^{g_2} & \cdots\ar[r]^{g_{n-3}}& Y_{n-2} \ar[r]^{g_{n-2}} & Y_{n-1} \\
}$$ in a pre-$n$-angulated category $\mathcal{C}$ is called a $homotopy\ cartesian\ diagram$ if the following $n$-$\Sigma$-sequence
$$\xymatrixcolsep{4.5pc}\xymatrix{
X_1\ar[r]^{\left(
             \begin{smallmatrix}
               -f_1 \\
               \varphi_1 \\
             \end{smallmatrix}
           \right)
}& X_2\oplus Y_1\ar[r]^{\left(
             \begin{smallmatrix}
              -f_2 & 0 \\
               \varphi_2 & g_1 \\
             \end{smallmatrix}
           \right)}& X_3\oplus Y_2\ar[r]^{\left(
             \begin{smallmatrix}
              -f_3 & 0 \\
               \varphi_3 & g_2 \\
             \end{smallmatrix}
           \right)}& \cdots\\
}$$
$$\begin{gathered}\xymatrixcolsep{5pc}\xymatrix{
\cdots\ar[r]^{\left(
             \begin{smallmatrix}
              -f_{n-2} & 0 \\
               \varphi_{n-2} & g_{n-3} \\
             \end{smallmatrix}
           \right)\ \ \ \ }& X_{n-1}\oplus Y_{n-2}\ar[r]^{\ \ \ \left(
             \begin{smallmatrix}
               \varphi_{n-1} & g_{n-2} \\
             \end{smallmatrix}
           \right)}& Y_{n-1} \ar[r]^\partial& \Sigma X_1
}\end{gathered}\eqno{(2.2)}$$
is an $n$-angle, where $\partial$ is called a $differential$.
\end{defn}

\begin{rem}\label{2.1}

(a) The $n$-angle (2.2) in the definition of homotopy cartesian diagram is slightly different from the one in the definition given in \cite{[BT1]}. But the two $n$-angles are isomorphic:
$$\xymatrixcolsep{4.5pc}\xymatrixrowsep{2.7pc}\xymatrix{
X_1\ar[r]^{\left(
             \begin{smallmatrix}
               -f_1 \\
               \varphi_1 \\
             \end{smallmatrix}
           \right)
}\ar@{=}[d]& X_2\oplus Y_1\ar[r]^{\left(
             \begin{smallmatrix}
              -f_2 & 0 \\
               \varphi_2 & g_1 \\
             \end{smallmatrix}
           \right)}\ar@{=}[d]& X_3\oplus Y_2\ar[r]^{\left(
             \begin{smallmatrix}
              -f_3 & 0 \\
               \varphi_3 & g_2 \\
             \end{smallmatrix}
           \right)}\ar[d]^{\left(
             \begin{smallmatrix}
              -1 & 0 \\
               0 & 1 \\
             \end{smallmatrix}
           \right)}& \cdots\\
X_1\ar[r]^{\left(
             \begin{smallmatrix}
               -f_1 \\
               \varphi_1 \\
             \end{smallmatrix}
           \right)
}& X_2\oplus Y_1\ar[r]^{\left(
             \begin{smallmatrix}
              f_2 & 0 \\
               \varphi_2 & g_1 \\
             \end{smallmatrix}
           \right)}& X_3\oplus Y_2\ar[r]^{\left(
             \begin{smallmatrix}
              f_3 & 0 \\
               -\varphi_3 & g_2 \\
             \end{smallmatrix}
           \right)}& \cdots\\
}$$
$$\xymatrixcolsep{5.5pc}\xymatrixrowsep{2.5pc}\xymatrix{
\cdots\ar[r]^{\left(
             \begin{smallmatrix}
              -f_{n-2} & 0 \\
               \varphi_{n-2} & g_{n-3} \\
             \end{smallmatrix}
           \right)}& X_{n-1}\oplus Y_{n-2}\ar[r]^{\left(
             \begin{smallmatrix}
               \varphi_{n-1} & g_{n-2} \\
             \end{smallmatrix}
           \right)}\ar[d]^{\left(
             \begin{smallmatrix}
              (-1)^{n+1} & 0 \\
              0 & 1 \\
             \end{smallmatrix}
           \right)} & Y_{n-1} \ar[r]^\partial \ar@{=}[d] & \Sigma X_1 \ar@{=}[d]\\
\cdots\ar[r]^{\left(
             \begin{smallmatrix}
              f_{n-2} & 0 \\
               (-1)^n\varphi_{n-2} & g_{n-3} \\
             \end{smallmatrix}
           \right)}& X_{n-1}\oplus Y_{n-2}\ar[r]^{\left(
             \begin{smallmatrix}
               (-1)^{n+1}\varphi_{n-1} & g_{n-2} \\
             \end{smallmatrix}
           \right)}& Y_{n-1} \ar[r]^\partial& \Sigma X_1
}$$

(b) Since a homotopy cartesian square is the triangulated analogue of a pullback and pushout square in an abelian category, we can compare our definition with the notion of $n$-pushout and $n$-pullback diagram in $n$-abelian categories \cite{[J]}.
\end{rem}

\begin{lem}\label{1}
Let $$\xymatrix{
X_\bullet\ar[d]^{\varphi_\bullet} & X_1 \ar[r]^{f_1}\ar@{=}[d] & X_2 \ar[r]^{f_2}\ar[d]^{\varphi_2} & X_3 \ar[r]^{f_3}\ar[d]^{\varphi_3} & \cdots \ar[r]^{f_{n-1}} & X_n \ar[r]^{f_n} \ar[d]^{\varphi_{n}} & \Sigma X_1 \ar@{=}[d]\\
Y_\bullet & X_1 \ar[r]^{g_1} & Y_2 \ar[r]^{g_2} & Y_3 \ar[r]^{g_3} & \cdots \ar[r]^{g_{n-1}} & Y_n \ar[r]^{g_n}& \Sigma X_1\\
}$$ be a commutative diagram where each row is an $n$-angle. Then as an $n$-$\Sigma$-sequence, the mapping cone $C(\varphi_\bullet)$ is isomorphic to the direct sum of $Z_\bullet$ and $X'_\bullet$, where
$$Z_\bullet=(X_2\xrightarrow{\left(
             \begin{smallmatrix}
               -f_2\\
                \varphi_2\\
             \end{smallmatrix}
           \right)}X_3\oplus Y_2\xrightarrow{\left(
             \begin{smallmatrix}
               -f_3 & 0 \\
               \varphi_3 & g_2\\
             \end{smallmatrix}
           \right)}\cdots\xrightarrow{\left(
             \begin{smallmatrix}
               -f_{n-1} & 0 \\
               \varphi_{n-1} & g_{n-2}\\
             \end{smallmatrix}
           \right)} X_n\oplus Y_{n-1}$$
$$\xrightarrow{(\varphi_n\ g_{n-1})} Y_n\xrightarrow{\Sigma f_1\cdot g_n}\Sigma X_2), \ \ X_\bullet'=(X_1\rightarrow 0\rightarrow\cdots\rightarrow 0\rightarrow \Sigma X_1\xrightarrow{1} \Sigma X_1).$$
In particular, $C(\varphi_\bullet)$ is exact if and only if $Z_\bullet$ is exact.
\end{lem}

\begin{proof}
It is easy to check that we have the following commutative diagram
$$\xymatrixcolsep{4pc}\xymatrixrowsep{2.5pc}\xymatrix{
X_2\oplus X_1\ar[r]^{\left(
             \begin{smallmatrix}
               -f_2 & 0\\
                \varphi_2 & g_1\\
             \end{smallmatrix}
           \right)}\ar[d]^{\left(
             \begin{smallmatrix}
               1 & f_1\\
               0 & 1\\
             \end{smallmatrix}
           \right)} & X_3\oplus Y_2\ar[r]^{\left(
             \begin{smallmatrix}
               -f_3 & 0 \\
               \varphi_3 & g_2\\
             \end{smallmatrix}
           \right)}\ar@{=}[d] & \cdots \\
X_2\oplus X_1\ar[r]^{\left(
             \begin{smallmatrix}
               -f_2 & 0\\
                \varphi_2 & 0\\
             \end{smallmatrix}
           \right)} & X_3\oplus Y_2\ar[r]^{\left(
             \begin{smallmatrix}
               -f_3 & 0 \\
               \varphi_3 & g_2\\
             \end{smallmatrix}
           \right)} & \cdots \\ }$$

$$\xymatrixcolsep{4.2pc}\xymatrixrowsep{2.5pc}\xymatrix{ \cdots\ar[r]^{\left(
             \begin{smallmatrix}
               -f_{n-1} & 0 \\
               \varphi_{n-1} & g_{n-2}\\
             \end{smallmatrix}
           \right)\ \ \ \ } &  X_n\oplus Y_{n-1}\ar[r]^{\left(
             \begin{smallmatrix}
               -f_n & 0 \\
               \varphi_n & g_{n-1}\\
             \end{smallmatrix}
           \right)}\ar@{=}[d] & \Sigma X_1\oplus Y_n\ar[r]^{\left(
             \begin{smallmatrix}
               -\Sigma f_1 & 0 \\
               1 & g_n\\
             \end{smallmatrix}
           \right)}\ar[d]^{\left(
             \begin{smallmatrix}
              0 & 1 \\
               1 & g_n\\
             \end{smallmatrix}
           \right)} & \Sigma X_2\oplus\Sigma X_1\ar[d]^{\left(
             \begin{smallmatrix}
              1 & \Sigma f_1 \\
               0 & 1\\
             \end{smallmatrix}
           \right)}\\
\cdots\ar[r]^{\left(
             \begin{smallmatrix}
               -f_{n-1} & 0 \\
               \varphi_{n-1} & g_{n-2}\\
             \end{smallmatrix}
           \right)\ \ \ \ } &  X_n\oplus Y_{n-1}\ar[r]^{\left(
             \begin{smallmatrix}
               \varphi_n & g_{n-1}\\
               0& 0\\
             \end{smallmatrix}
           \right)} & Y_n\oplus\Sigma X_1\ar[r]^{\left(
             \begin{smallmatrix}
               \Sigma f_1\cdot g_n & 0 \\
               0 & 1\\
             \end{smallmatrix}
           \right)} & \Sigma X_2\oplus\Sigma X_1\\
 }$$ whose second row is the direct sum of $Z_\bullet$ and $X'_\bullet$. The second assertion is clear.
\end{proof}

\begin{lem} (\cite[Lemma 2.4, Lemma 2.5]{[GKO]}) \label{lem0}
Let $(\mathcal{C},\Sigma, \Theta)$ be a pre-$n$-angulated category, then the following hold.

(a) All $n$-angles are exact.

(b) Let $\varphi_\bullet: X_\bullet\rightarrow Y_\bullet$ be a weak isomorphism of exact $n$-$\Sigma$-sequences. Then $X_\bullet$ is an $n$-angle if and only if $Y_\bullet$ is an $n$-angle.
\end{lem}

\begin{lem}\label{2}
Let $$\xymatrix{
X_\bullet\ar[d]^{\varphi_\bullet} & X_1 \ar[r]^{f_1}\ar@{=}[d] & X_2 \ar[r]^{f_2}\ar[d]^{\varphi_2} & X_3 \ar[r]^{f_3}\ar[d]^{\varphi_3} & \cdots \ar[r]^{f_{n-1}} & X_n \ar[r]^{f_n} \ar[d]^{\varphi_{n}} & \Sigma X_1 \ar@{=}[d]\\
Y_\bullet\ar[d]^{\psi_\bullet} & X_1 \ar[r]^{g_1}\ar@{=}[d] & Y_2 \ar[r]^{g_2}\ar@{=}[d] & Y_3 \ar[r]^{g_3}\ar[d]^{\psi_3} & \cdots \ar[r]^{g_{n-1}} & Y_n \ar[r]^{g_n}\ar[d]^{\psi_n} & \Sigma X_1\ar@{=}[d]\\
Z_\bullet & X_1 \ar[r]^{g_1} & Y_2 \ar[r]^{h_2} & Z_3 \ar[r]^{h_3} & \cdots \ar[r]^{h_{n-1}} & Z_n \ar[r]^{h_n}& \Sigma X_1
}$$ be a commutative diagram whose rows are $n$-angles.
Then $$\xymatrix{
 X_2 \ar[r]^{f_2}\ar[d]^{\varphi_2} & X_3 \ar[r]^{f_3}\ar[d]^{\varphi_3} & \cdots \ar[r]^{f_{n-1}}& X_{n} \ar[d]^{\varphi_{n}} \\
 Y_2 \ar[r]^{g_2} & Y_3 \ar[r]^{g_3} & \cdots\ar[r]^{g_{n-1}}&  Y_{n} \\
}$$ is a homotopy cartesian diagram with the differential $\Sigma f_1\cdot g_n$ if and only if $$\xymatrix{
 X_2 \ar[r]^{f_2}\ar[d]^{\varphi_2} & X_3 \ar[r]^{f_3}\ar[d]^{\psi_3\varphi_3} & \cdots \ar[r]^{f_{n-1}}& X_{n} \ar[d]^{\psi_n\varphi_{n}} \\
 Y_2 \ar[r]^{h_2} & Z_3 \ar[r]^{h_3} & \cdots\ar[r]^{h_{n-1}}&  Z_{n} \\
}$$
is a homotopy cartesian diagram with the differential $\Sigma f_1\cdot h_n$.
\end{lem}

\begin{proof} By Lemma \ref{lem0}(a), the $n$-angles $X_\bullet, Y_\bullet$ and $Z_\bullet$ are exact.
Since the class of exact $n$-$\Sigma$-sequences is closed under mapping cones, we obtain that the mapping cones $C(\varphi_\bullet)$ and $C(\psi_\bullet\varphi_\bullet)$ are also exact. It is easy to see that we have the following commutative diagram
$$\xymatrixcolsep{3.5pc}\xymatrixrowsep{2.5pc}\xymatrix{
X_2\ar[r]^{\left(
             \begin{smallmatrix}
               -f_2\\
                \varphi_2\\
             \end{smallmatrix}
           \right)}\ar@{=}[d] & X_3\oplus Y_2\ar[r]^{\left(
             \begin{smallmatrix}
               -f_3 & 0 \\
               \varphi_3 & g_2\\
             \end{smallmatrix}
           \right)}\ar@{=}[d] & X_4\oplus Y_3\ar[r]^{\left(
             \begin{smallmatrix}
               -f_4 & 0 \\
               \varphi_4 & g_3\\
             \end{smallmatrix}
           \right)}\ar[d]^{\left(
             \begin{smallmatrix}
               1 & 0 \\
               0 & \psi_3\\
             \end{smallmatrix}
           \right)} & \cdots\\
X_2\ar[r]^{\left(
             \begin{smallmatrix}
               -f_2\\
                \varphi_2\\
             \end{smallmatrix}
           \right)} & X_3\oplus Y_2\ar[r]^{\left(
             \begin{smallmatrix}
               -f_3 & 0 \\
               \psi_3\varphi_3 & h_2\\
             \end{smallmatrix}
           \right)} & X_4\oplus Z_3\ar[r]^{\left(
             \begin{smallmatrix}
               -f_4 & 0 \\
               \psi_4\varphi_4 & h_3\\
             \end{smallmatrix}
           \right)} & \cdots\\
}$$
$$\xymatrixcolsep{4pc}\xymatrixrowsep{2.5pc}\xymatrix{
 \cdots\ar[r]^{\left(
             \begin{smallmatrix}
               -f_{n-1} & 0 \\
               \varphi_{n-1} & g_{n-2}\\
             \end{smallmatrix}
           \right)\ \ \ \ } & X_n\oplus Y_{n-1}\ar[r]^{\ \ (\varphi_n,g_{n-1})}\ar[d]^{\left(
             \begin{smallmatrix}
               1 & 0 \\
               0 & \psi_{n-1}\\
             \end{smallmatrix}
           \right)} & Y_n\ar[r]^{\Sigma f_1\cdot g_n}\ar[d]^{\psi_n} & \Sigma X_2 \ar@{=}[d]\\
 \cdots\ar[r]^{\left(
             \begin{smallmatrix}
               -f_{n-1} & 0 \\
               \psi_{n-1}\varphi_{n-1} & h_{n-2}\\
             \end{smallmatrix}
           \right)\ \ \ \ \ \ } & X_n\oplus Z_{n-1}\ar[r]^{\ \ \ (\psi_n\varphi_n,h_{n-1})} & Z_n\ar[r]^{\Sigma f_1\cdot h_n} & \Sigma X_2\\
}$$
whose rows are exact $n$-$\Sigma$-sequences by Lemma \ref{1}. The result holds by definition and Lemma \ref{lem0}(b).
\end{proof}

\section{equivalent statements of higher mapping cone axiom}

In this section, we develop some equivalent statements of higher mapping cone axiom to explain the higher octahedral axiom. We leave the dual statements to the reader.

\begin{thm}\label{thm1}
Let $(\mathcal{C},\Sigma,\Theta)$ be a pre-$n$-angulated category. Then $\Theta$ satisfies (N4) if and only if $\Theta$ satisfies (N4-1):

Let  $$\xymatrix{
X_1 \ar[r]^{f_1}\ar@{=}[d] & X_2 \ar[r]^{f_2}\ar[d]^{\varphi_2} & X_3 \ar[r]^{f_3} & \cdots \ar[r]^{f_{n-1}}& X_n \ar[r]^{f_n} & \Sigma X_1 \ar@{=}[d]\\
X_1 \ar[r]^{g_1} & Y_2 \ar[r]^{g_2} & Y_3 \ar[r]^{g_3} & \cdots \ar[r]^{g_{n-1}} & Y_n \ar[r]^{g_n}& \Sigma X_1\\
}$$ be a commutative diagram whose rows are $n$-angles. Then there exist morphisms $\varphi_i:X_i\rightarrow Y_i$ for $3\leq i\leq n$ such that the following
$$\xymatrix{
 X_2 \ar[r]^{f_2}\ar[d]^{\varphi_2} & X_3 \ar[r]^{f_3}\ar[d]^{\varphi_3} & \cdots \ar[r]^{f_{n-1}}& X_n \ar[d]^{\varphi_n} \\
 Y_2 \ar[r]^{g_2} & Y_3 \ar[r]^{g_3} & \cdots \ar[r]^{g_{n-1}} & Y_n \\
}$$ is a homotopy cartesian diagram and $\Sigma f_1\cdot g_n$ is the differential.

\end{thm}

\begin{proof}
Assume that $\Theta$ satisfies (N4), then there exist morphisms $\varphi_i: X_i\rightarrow Y_i$ for $3\leq i\leq n$ such that the mapping cone $C(\varphi_\bullet)$ is an $n$-angle. Since the class of $n$-angles is closed under direct summands, the remaining part of (N4-1) follows from Lemma \ref{1}.

Conversely, we assume that $\Theta$ satisfies (N4-1). Given a commutative diagram
$$\xymatrix{
X_1 \ar[r]^{f_1}\ar[d]^{\varphi_1} & X_2 \ar[r]^{f_2}\ar[d]^{\varphi_2} & X_3 \ar[r]^{f_3} & \cdots \ar[r]^{f_{n-1}}& X_n \ar[r]^{f_n} & \Sigma X_1 \ar[d]^{\Sigma \varphi_1}\\
Y_1 \ar[r]^{g_1} & Y_2 \ar[r]^{g_2} & Y_3 \ar[r]^{g_3} & \cdots \ar[r]^{g_{n-1}} & Y_n \ar[r]^{g_n}& \Sigma Y_1\\
}$$  whose rows are $n$-angles.  The following commutative diagram
$$\xymatrixcolsep{3.5pc}\xymatrix{
 X_1\oplus Y_1 \ar[r]^{\left(
                    \begin{smallmatrix}
                      \varphi_2f_1& g_1 \\
                    \end{smallmatrix}
                  \right)} \ar[d]^{\left(
                    \begin{smallmatrix}
                      1 & 0 \\
                      \varphi_1 & 1\\
                    \end{smallmatrix}
                  \right)} & Y_2\ar[r]^{g_2}\ar@{=}[d] & Y_3\ar[r]^{g_3}\ar@{=}[d] &\cdots\\
        X_1\oplus Y_1 \ar[r]^{\left(
                    \begin{smallmatrix}
                      0& g_1 \\
                    \end{smallmatrix}
                  \right)}  & Y_2\ar[r]^{g_2} & Y_3\ar[r]^{g_3} & \cdots\\
}$$
$$\xymatrixcolsep{4pc}\xymatrix{
\cdots\ar[r]^{g_{n-2}} & Y_{n-1} \ar[r]^{\left(
                    \begin{smallmatrix}
                      0 \\
                      g_{n-1} \\
                    \end{smallmatrix}
                  \right)}\ar@{=}[d]& \Sigma X_1\oplus Y_n \ar[r]^{\tiny\left(
                    \begin{smallmatrix}
                     -1 & 0 \\
                     \Sigma \varphi_1 & g_n \\
                    \end{smallmatrix}
                  \right)} \ar@{=}[d]& \Sigma X_1\oplus\Sigma Y_1 \ar[d]^{\left(
                    \begin{smallmatrix}
                      1& 0 \\
                      \Sigma \varphi_1 & 1 \\
                    \end{smallmatrix}
                  \right)}\\
\cdots\ar[r]^{g_{n-2}} & Y_{n-1} \ar[r]^{\left(
                    \begin{smallmatrix}
                      0 \\
                      g_{n-1} \\
                    \end{smallmatrix}
                  \right)}& \Sigma X_1\oplus Y_n \ar[r]^{\left(
                    \begin{smallmatrix}
                     -1 & 0 \\
                     0 & g_n \\
                    \end{smallmatrix}
                  \right)}& \Sigma X_1\oplus\Sigma Y_1\\
}$$
shows that the first row  is an $n$-angle since the second row is a direct sum of two $n$-angles. By (N4-1), the following diagram
$$\xymatrixcolsep{3.5pc}\xymatrix{
X_1\oplus Y_1\ar[r]^{\left(
                    \begin{smallmatrix}
                      f_1 & 0 \\
                      0 & 1 \\
                    \end{smallmatrix}
                  \right)}\ar@{=}[d] &  X_2\oplus Y_1 \ar[r]^{\left(
                    \begin{smallmatrix}
                      f_2 & 0 \\
                    \end{smallmatrix}
                  \right)}\ar[d]^{\left(
                    \begin{smallmatrix}
                       \varphi_2 & g_1 \\
                    \end{smallmatrix}
                  \right)}  & X_3 \ar[r]^{f_3}\ar@{-->}[d]^{\varphi_3}   & \cdots\\
        X_1\oplus Y_1 \ar[r]^{\left(
                    \begin{smallmatrix}
                      \varphi_2f_1& g_1 \\
                    \end{smallmatrix}
                  \right)}  & Y_2\ar[r]^{g_2} & Y_3\ar[r]^{g_3} &\cdots\\
}$$
$$\xymatrixcolsep{4pc}\xymatrix{
\cdots \ar[r]^{f_{n-2}} & X_{n-1} \ar[r]^{f_{n-1}}\ar@{-->}[d]^{\varphi_{n-1}} & X_n \ar[r]^{\left(
                    \begin{smallmatrix}
                      f_n \\
                      0 \\
                    \end{smallmatrix}
                  \right)} \ar@{-->}[d]^{\left(
                                     \begin{smallmatrix}
                                       \varphi_n' \\
                                       \varphi_n \\
                                     \end{smallmatrix}
                                   \right)}  & \Sigma X_1\oplus\Sigma Y_1 \ar@{=}[d] \\
\cdots\ar[r]^{g_{n-2}} & Y_{n-1} \ar[r]^{\left(
                    \begin{smallmatrix}
                      0 \\
                      g_{n-1} \\
                    \end{smallmatrix}
                  \right)} & \Sigma X_1\oplus Y_n \ar[r]^{\tiny\left(
                    \begin{smallmatrix}
                     -1 & 0 \\
                     \Sigma \varphi_1 & g_n \\
                    \end{smallmatrix}
                  \right)} & \Sigma X_1\oplus\Sigma Y_1 \\
}$$ can be completed to a morphism of $n$-angles
 such that the following sequence
$$\xymatrixcolsep{4pc}\xymatrix{
 X_2\oplus Y_1\ar[r]^{\left(
                            \begin{smallmatrix}
                                -f_2 & 0 \\
                                \varphi_2 & g_1\\
                             \end{smallmatrix}
                           \right)} & X_3\oplus Y_2 \ar[r]^{\left(
                            \begin{smallmatrix}
                               -f_3  & 0 \\
                                \varphi_3  & g_2 \\
                             \end{smallmatrix}
                           \right)} & X_4\oplus Y_3 \ar[r]^{\left(
                            \begin{smallmatrix}
                               -f_4 & 0 \\
                                \varphi_4 & g_{3} \\
                             \end{smallmatrix}
                           \right)} & \cdots \\
}$$
$$\xymatrixcolsep{4.5pc}\xymatrix{
\cdots\ar[r]^{\left(
                            \begin{smallmatrix}
                               -f_{n-1} & 0 \\
                                \varphi_{n-1} & g_{n-2} \\
                             \end{smallmatrix}
                           \right)\ \ \ \ } & X_n\oplus Y_{n-1}\ar[r]^{\left(
                            \begin{smallmatrix}
                               \varphi'_n & 0 \\
                                \varphi_n & g_{n-1} \\
                             \end{smallmatrix}
                           \right)} & \Sigma X_1\oplus Y_n \ar[r]^{\left(
                            \begin{smallmatrix}
                               -\Sigma f_1 & 0 \\
                                \Sigma\varphi_1 & g_{n} \\
                             \end{smallmatrix}
                           \right)} & \Sigma X_2\oplus\Sigma Y_1\\
}$$ is an $n$-angle, where $\varphi_n'=-f_n$ by the commutativity of the above rightmost square.
\end{proof}

\begin{thm}\label{thm2}
Let $(\mathcal{C},\Sigma,\Theta)$ be a pre-$n$-angulated category. Then the following statements are equivalent.

(a) $\Theta$ satisfies (N4-1).

(b) $\Theta$ satisfies (N4-2):

Given an $n$-angle $X_1\xrightarrow{f_1} X_2\xrightarrow{f_2} \cdots \xrightarrow{f_{n-1}} X_n\xrightarrow{f_n}\Sigma X_1$  and a morphism $\varphi_1:X_1\rightarrow Y_1$, there exists a commutative diagram
$$\xymatrix{
X_1 \ar[r]^{f_1}\ar[d]^{\varphi_1} & X_2 \ar[r]^{f_2}\ar[d]^{\varphi_2} & \cdots \ar[r]^{f_{n-2}}& X_{n-1} \ar[r]^{f_{n-1}}\ar[d]^{\varphi_{n-1}}& X_n \ar[r]^{f_n}\ar@{=}[d] & \Sigma X_1\ar[d]^{\Sigma\varphi_1}\\
Y_1 \ar[r]^{g_1} & Y_2 \ar[r]^{g_2} & \cdots \ar[r]^{g_{n-2}} &Y_{n-1} \ar[r]^{g_{n-1}} & X_n \ar[r]^{g_n}& \Sigma Y_1\\
}$$
such that the second row is an $n$-angle and $$\xymatrix{
 X_1 \ar[r]^{f_1}\ar[d]^{\varphi_1} & X_2 \ar[r]^{f_2}\ar[d]^{\varphi_2} & \cdots \ar[r]^{f_{n-2}}& X_{n-1} \ar[d]^{\varphi_{n-1}} \\
 Y_1 \ar[r]^{g_1} & Y_2 \ar[r]^{g_2} & \cdots \ar[r]^{g_{n-2}} & Y_{n-1} \\
}$$ is a homotopy cartesian diagram where $(-1)^nf_n\cdot g_{n-1}$ is the differential.

(c) $\Theta$ satisfies (N4-3):

 Given two morphisms $f_1 :X_1\rightarrow X_2$ and $\varphi_2:X_2\rightarrow Y_2$, there exists a commutative diagram
$$\xymatrix{
X_1 \ar[r]^{f_1}\ar@{=}[d] & X_2 \ar[r]^{f_2}\ar[d]^{\varphi_2} & X_3 \ar[r]^{f_3}\ar[d]^{\varphi_3} & \cdots \ar[r]^{f_{n-1}}& X_n \ar[r]^{f_n}\ar[d]^{\varphi_n} & \Sigma X_1 \ar@{=}[d]\\
X_1 \ar[r]^{g_1} & Y_2 \ar[r]^{g_2} & Y_3 \ar[r]^{g_3} & \cdots \ar[r]^{g_{n-1}} & Y_n \ar[r]^{g_n}& \Sigma X_1\\
}$$
such that each row is an $n$-angle and $$\xymatrix{
 X_2 \ar[r]^{f_2}\ar[d]^{\varphi_2} & X_3 \ar[r]^{f_3}\ar[d]^{\varphi_3} & \cdots \ar[r]^{f_{n-1}}& X_n \ar[d]^{\varphi_n} \\
 Y_2 \ar[r]^{g_2} & Y_3 \ar[r]^{g_3} & \cdots \ar[r]^{g_{n-1}} & Y_n \\
}$$ is a homotopy cartesian diagram where $\Sigma f_1\cdot g_n$ is the differential.

\end{thm}

\begin{proof}
(a) $\Rightarrow$ (b). By (N2) and (N1)(c), we have the following commutative diagram
$$\xymatrixcolsep{3.3pc}\xymatrix{
\Sigma^{-1}X_n\ar[r]^{(-1)^n\Sigma^{-1}f_n}\ar@{=}[d] & X_1 \ar[r]^{f_1}\ar[d]^{\varphi_1} & X_2 \ar[r]^{f_2}  & \cdots \ar[r]^{f_{n-2}} & X_{n-1}\ar[r]^{f_{n-1}} & X_n  \ar@{=}[d]\\
\Sigma^{-1}X_n\ar[r]^{(-1)^n\varphi_1\Sigma^{-1}f_n} & Y_1 \ar[r]^{g_1} & Y_2 \ar[r]^{g_2}  & \cdots \ar[r]^{g_{n-2}}& Y_{n-1} \ar[r]^{g_{n-1}} & X_n \\
}$$ whose rows are $n$-angles. Now (b) follows from (N4-1).

(b) $\Rightarrow$ (c). For the morphism $f_1: X_1\rightarrow X_2$, by (N1)(c) and (N2) we assume that $ X_2\xrightarrow{f_2} \cdots \xrightarrow{f_{n-1}} X_n\xrightarrow{f_n}\Sigma X_1\xrightarrow{(-1)^n\Sigma f_1} \Sigma X_2$ is an $n$-angle. Then (c) follows from (N4-2).

(c) $\Rightarrow$ (a).  Given a commutative diagram  $$\xymatrix{
X_1 \ar[r]^{f_1}\ar@{=}[d] & X_2 \ar[r]^{f_2}\ar[d]^{\varphi_2} & X_3 \ar[r]^{f_3} & \cdots \ar[r]^{f_{n-1}}& X_n \ar[r]^{f_n} & \Sigma X_1 \ar@{=}[d]\\
X_1 \ar[r]^{g_1} & Y_2 \ar[r]^{g_2} & Y_3 \ar[r]^{g_3} & \cdots \ar[r]^{g_{n-1}} & Y_n \ar[r]^{g_n}& \Sigma X_1\\
}$$  whose rows are $n$-angles. By (c), there exists a commutative diagram
$$\xymatrix{
X_1 \ar[r]^{f_1}\ar@{=}[d] & X_2 \ar[r]^{f'_2}\ar[d]^{\varphi_2} & X'_3 \ar[r]^{f'_3}\ar[d]^{\varphi'_3} & \cdots \ar[r]^{f'_{n-1}}& X'_n \ar[r]^{f'_n}\ar[d]^{\varphi'_n} & \Sigma X_1 \ar@{=}[d]\\
X_1 \ar[r]^{g_1} & Y_2 \ar[r]^{g'_2} & Y'_3 \ar[r]^{g'_3} & \cdots \ar[r]^{g'_{n-1}} & Y'_n \ar[r]^{g'_n}& \Sigma X_1\\
}$$
such that each row is an $n$-angle and $$\xymatrix{
 X_2 \ar[r]^{f'_2}\ar[d]^{\varphi_2} & X'_3 \ar[r]^{f'_3}\ar[d]^{\varphi'_3} & \cdots \ar[r]^{f'_{n-1}}& X'_n \ar[d]^{\varphi'_n} \\
 Y_2 \ar[r]^{g'_2} & Y'_3 \ar[r]^{g'_3} & \cdots \ar[r]^{g'_{n-1}} & Y'_n \\
}$$ is a homotopy cartesian diagram where $\Sigma f_1\cdot g'_n$ is the differential. It follows from (N3) that we have the following commutative diagram
$$\xymatrix{
X_\bullet\ar[d]^{\theta_\bullet} & X_1 \ar[r]^{f_1}\ar@{=}[d] & X_2 \ar[r]^{f_2}\ar@{=}[d] & X_3 \ar[r]^{f_3}\ar[d]^{\theta_3} & \cdots \ar[r]^{f_{n-1}}& X_n \ar[r]^{f_n}\ar[d]^{\theta_n} & \Sigma X_1 \ar@{=}[d]\\
X'_\bullet\ar[d]^{\varphi'_\bullet} & X_1 \ar[r]^{f_1}\ar@{=}[d] & X_2 \ar[r]^{f'_2}\ar[d]^{\varphi_2} & X'_3 \ar[r]^{f'_3}\ar[d]^{\varphi'_3} & \cdots \ar[r]^{f'_{n-1}}& X'_n \ar[r]^{f'_n}\ar[d]^{\varphi'_n} & \Sigma X_1 \ar@{=}[d]\\
Y'_\bullet\ar[d]^{\psi_\bullet}& X_1 \ar[r]^{g_1}\ar@{=}[d] & Y_2 \ar[r]^{g'_2}\ar@{=}[d] & Y'_3 \ar[r]^{g'_3}\ar[d]^{\psi_3} & \cdots \ar[r]^{g'_{n-1}} & Y'_n \ar[r]^{g'_n} \ar[d]^{\psi_n}& \Sigma X_1\ar@{=}[d]\\
Y_\bullet & X_1 \ar[r]^{g_1} & Y_2 \ar[r]^{g_2} & Y_3 \ar[r]^{g_3} & \cdots \ar[r]^{g_{n-1}} & Y_n \ar[r]^{g_n}& \Sigma X_1\\
}$$ 
whose rows are $n$-angles. Lemma \ref{2} and its dual imply that
$$\xymatrix{
 X_2 \ar[r]^{f_2}\ar[d]^{\varphi_2} & X_3 \ar[r]^{f_3}\ar[d]^{\psi_3\varphi'_3\theta_3} & \cdots \ar[r]^{f_{n-1}}& X_n \ar[d]^{\psi_n\varphi'_n\theta_n} \\
 Y_2 \ar[r]^{g_2} & Y_3 \ar[r]^{g_3} & \cdots \ar[r]^{g_{n-1}} & Y_n \\
}$$ is a homotopy cartesian diagram and $\Sigma f_1\cdot g_n$ is the differential.
\end{proof}

\begin{rem}
The axiom (N4-1) and axiom (N4-2) are the higher analogues of  homotopy cartesian axiom and cobase change axiom respectively. We can see \cite{[Ne],[Kr],[Be]} for reference.
\end{rem}

Now we will use the higher homotopy cartesian axiom (N4-1) to explain the higher octahedral axiom (N4$^*$).

\begin{thm}\label{thm3}
Let $(\mathcal{C},\Sigma,\Theta)$ be a pre-$n$-angulated category. Then the following statements are equivalent.

 (a) $\Theta$ satisfies (N4-1).

 (b) $\Theta$ satisfies (N4-4):

 Given the following commutative diagram
 $$\xymatrix{
X_1 \ar[r]^{f_1}\ar@{=}[d] & X_2 \ar[r]^{f_2}\ar[d]^{\varphi_2} & X_3 \ar[r]^{f_3} & \cdots \ar[r]^{f_{n-1}}& X_n \ar[r]^{f_n} & \Sigma X_1 \ar@{=}[d]\\
X_1 \ar[r]^{g_1}\ar[d]^{f_1} & Y_2 \ar[r]^{g_2}\ar@{=}[d] & Y_3 \ar[r]^{g_3} & \cdots \ar[r]^{g_{n-1}} & Y_n \ar[r]^{g_n}& \Sigma X_1\ar[d]^{\Sigma f_1}\\
X_2 \ar[r]^{\varphi_2} & Y_2 \ar[r]^{h_2} & Z_3 \ar[r]^{h_3} & \cdots \ar[r]^{h_{n-1}} & Z_n \ar[r]^{h_n}& \Sigma X_2\\
}$$ whose rows are $n$-angles, there exist morphisms $\varphi_i: X_i\rightarrow Y_i$ for $3\leq i\leq n$, $\psi_j:Y_j\rightarrow Z_j$ for $3\leq j\leq n$ and $\theta_k:X_k\rightarrow Z_{k-1}$ for $4\leq k\leq n$ such that the following diagram
$$\xymatrixcolsep{4pc}\xymatrix{
X_1 \ar[r]^{f_1}\ar@{=}[d] & X_2 \ar[r]^{f_2}\ar[d]^{\varphi_2} & X_3 \ar[r]^{f_3} \ar@{-->}[d]^{\varphi_3} & X_4 \ar[r]^{f_4} \ar@{-->}[d]^{\varphi_4} &\cdots \\
X_1 \ar[r]^{g_1}\ar[d]^{f_1} & Y_2 \ar[r]^{g_2}\ar[d]^{\left(
             \begin{smallmatrix}
               0 \\
               1 \\
             \end{smallmatrix}
           \right)} & Y_3 \ar[r]^{g_3}\ar[d]^{\left(
             \begin{smallmatrix}
               0 \\
               1 \\
             \end{smallmatrix}
           \right)} & Y_4 \ar[r]^{g_4} \ar[d]^{\left(
             \begin{smallmatrix}
               0 \\
               1 \\
             \end{smallmatrix}
           \right)} &\cdots \\
X_2\ar[r]^{\left(
             \begin{smallmatrix}
               -f_2 \\
               \varphi_2 \\
             \end{smallmatrix}
           \right)
} \ar@{=}[d] & X_3\oplus Y_2\ar[r]^{\left(
             \begin{smallmatrix}
              -f_3 & 0 \\
               \varphi_3 & g_2 \\
             \end{smallmatrix}
           \right)}\ar[d]^{\left(
             \begin{smallmatrix}
             0 & 1 \\
             \end{smallmatrix}
           \right)} & X_4\oplus Y_3 \ar[r]^{\left(
             \begin{smallmatrix}
              -f_4 & 0 \\
               \varphi_4 & g_3 \\
             \end{smallmatrix}
           \right)} \ar@{-->}[d]^{\left(
             \begin{smallmatrix}
             \theta_4 & \psi_3 \\
             \end{smallmatrix}
           \right)} & X_5\oplus Y_4 \ar[r]^{\left(
             \begin{smallmatrix}
              -f_5 & 0 \\
               \varphi_5 & g_4 \\
             \end{smallmatrix}
           \right)} \ar@{-->}[d]^{\left(
             \begin{smallmatrix}
             \theta_5 & \psi_4 \\
             \end{smallmatrix}
           \right)} & \cdots \\
X_2 \ar[r]^{\varphi_2}\ar[d]^{-f_2} & Y_2 \ar[r]^{h_2} \ar[d]^{\left(
             \begin{smallmatrix}
             0 \\
             g_2\\
             \end{smallmatrix}
           \right)}& Z_3 \ar[r]^{h_3}\ar[d]^{\left(
             \begin{smallmatrix}
             0 \\
             0\\
              1 \\
             \end{smallmatrix}
           \right)} & Z_4 \ar[r]^{h_4}\ar[d]_{\left(
             \begin{smallmatrix}
             0 \\
             0\\
             1 \\
             \end{smallmatrix}
           \right)} & \cdots \\
X_3\ar[r]^{\left(
             \begin{smallmatrix}
             f_3 \\
             -\varphi_3\\
             \end{smallmatrix}
           \right)} & X_4\oplus Y_3 \ar[r]^{\left(
             \begin{smallmatrix}
             f_4 & 0 \\
             -\varphi_4 & -g_3\\
              \theta_4 & \psi_3\\
             \end{smallmatrix}
           \right)}  & X_5\oplus Y_4\oplus Z_3 \ar[r]^{\left(
             \begin{smallmatrix}
             f_5 & 0 & 0 \\
             -\varphi_5 & -g_4 & 0 \\
             \theta_5 & \psi_4 & h_3\\
             \end{smallmatrix}
           \right)}  & X_6\oplus Y_5\oplus Z_4 \ar[r]^{\left(
             \begin{smallmatrix}
             f_6 & 0 & 0 \\
             -\varphi_6 & -g_5 & 0 \\
             \theta_6 & \psi_5 & h_4\\
             \end{smallmatrix}
           \right)}  & \cdots \\
}$$
$$\begin{gathered}\xymatrixcolsep{4pc}\xymatrix{
 \cdots \ar[r]^{f_{n-3}} & X_{n-2} \ar[r]^{f_{n-2}} \ar@{-->}[d]^{\varphi_{n-2}}& X_{n-1}\ar[r]^{f_{n-1}}\ar@{-->}[d]^{\varphi_{n-1}}& X_n \ar[r]^{f_n}\ar@{-->}[d]^{\varphi_{n}} & \Sigma X_1 \ar@{=}[d]\\
 \cdots \ar[r]^{g_{n-3}} & Y_{n-2} \ar[r]^{g_{n-2}} \ar[d]^{\left(
             \begin{smallmatrix}
               0 \\
               1 \\
             \end{smallmatrix}
           \right)} & Y_{n-1} \ar[r]^{g_{n-1}}\ar[d]^{\left(
             \begin{smallmatrix}
               0 \\
               1 \\
             \end{smallmatrix}
           \right)} & Y_n \ar[r]^{g_n} \ar@{=}[d]& \Sigma X_1\ar[d]^{\Sigma f_1}\\
  \cdots \ar[r]^{\left(
             \begin{smallmatrix}
              -f_{n-2} & 0 \\
               \varphi_{n-2} & g_{n-3} \\
             \end{smallmatrix}
           \right)}& X_{n-1}\oplus Y_{n-2} \ar[r]^{\left(
             \begin{smallmatrix}
              -f_{n-1} & 0 \\
               \varphi_{n-1} & g_{n-2} \\
             \end{smallmatrix}
           \right)} \ar@{-->}[d]^{\left(
             \begin{smallmatrix}
               \theta_{n-1}& \psi_{n-2} \\
             \end{smallmatrix}
           \right)} & X_{n}\oplus Y_{n-1} \ar[r]^{(\varphi_n\ g_{n-1})}\ar@{-->}[d]^{\left(
             \begin{smallmatrix}
               \theta_{n}& \psi_{n-1} \\
             \end{smallmatrix}
           \right)} & Y_n \ar[r]^{\Sigma f_1\cdot g_n} \ar@{-->}[d]^{\psi_n} & \Sigma X_2 \ar@{=}[d] \\
 \cdots \ar[r]^{h_{n-3}} & Z_{n-2}\ar[r]^{h_{n-2}}\ar[d]^{\left(
             \begin{smallmatrix}
             0 \\
             0\\
             1\\
             \end{smallmatrix}
           \right)}& Z_{n-1}\ar[r]^{h_{n-1}}\ar[d]^{\left(
             \begin{smallmatrix}
             0 \\
             1\\
             \end{smallmatrix}
           \right)} & Z_n \ar[r]^{h_n}\ar@{=}[d]& \Sigma X_2 \ar[d]^{-\Sigma f_2}\\
 \cdots  \ar[r]^{\left(
             \begin{smallmatrix}
             f_{n-1} & 0 & 0 \\
             -\varphi_{n-1} & -g_{n-2} & 0 \\
             \theta_{n-1} & \psi_{n-2} & h_{n-3}\\
             \end{smallmatrix}
           \right)}  & X_n\oplus Y_{n-1}\oplus Z_{n-2} \ar[r]^{\left(
             \begin{smallmatrix}
             -\varphi_{n} & -g_{n-1} & 0 \\
             \theta_{n} & \psi_{n-1} & h_{n-2}\\
             \end{smallmatrix}
           \right)} & Y_n\oplus Z_{n-1} \ar[r]^{\left(
             \begin{smallmatrix}
             \psi_n & h_{n-1} \\
             \end{smallmatrix}
           \right)} & Z_n \ar[r]^{-\Sigma f_2\cdot h_n} & \Sigma X_3\\
    }\end{gathered}\eqno{(3.1)}$$
is commutative where each row is an $n$-angle.

(c) $\Theta$ satisfies (N4-5):

Given two  morphisms $f_1 :X_1\rightarrow X_2$ and $\varphi_2:X_2\rightarrow Y_2$, there exists a commutative diagram (3.1)
such that each row is an $n$-angle.
\end{thm}

\begin{proof}
(a) $\Rightarrow$ (b).  By (N4-1), there exist morphisms $\varphi_i: X_i\rightarrow Y_i$ for $3\leq i\leq n$ such that the diagram (3.1) involving the first two rows is commutative and the third row is the $n$-angle given by the homotopy cartesian diagram. By (N4-1) again, there exist morphisms $\psi_j:Y_j\rightarrow Z_j$ for $3\leq j\leq n$ and $\theta_k:X_k\rightarrow Z_{k-1}$ for $4\leq k\leq n$ such that the diagram involving the third row and the fourth row is commutative, and the $n$-angle given by the homotopy cartesian diagram is the direct sum of the fifth row and the trivial $n$-angle on $Y_2$. Other communicative squares are trivial.

(b) $\Rightarrow$ (a). It is clear.

Since axiom (N4-1) is equivalent to axiom (N4-3) by Theorem \ref{thm2}, the proof of (a) $\Leftrightarrow$ (b) implies (a) $\Leftrightarrow$ (c).
\end{proof}

\begin{rem}
If we replace $\theta_k$ with $(-1)^{k+1}\phi_k$ for $4\leq k\leq n$ in the last row of diagram (3.1), then we have the following  isomorphism of $n$-angles 
$$\xymatrixcolsep{4pc}\xymatrixrowsep{4.7pc}\xymatrix{
X_3\ar[r]^{\left(
             \begin{smallmatrix}
             f_3 \\
             -\varphi_3\\
             \end{smallmatrix}
           \right)}\ar@{=}[d] & X_4\oplus Y_3 \ar[r]^{\left(
             \begin{smallmatrix}
             f_4 & 0 \\
             -\varphi_4 & -g_3\\
              -\phi_4 & \psi_3\\
             \end{smallmatrix}
           \right)\ \ } \ar[d]^{\left(
             \begin{smallmatrix}
             1 & 0 \\
             0 & -1 \\
             \end{smallmatrix}
           \right)} & X_5\oplus Y_4\oplus Z_3 \ar[r]^{\left(
             \begin{smallmatrix}
             f_5 & 0 & 0 \\
             -\varphi_5 & -g_4 & 0 \\
             \phi_5 & \psi_4 & h_3\\
             \end{smallmatrix}
           \right)} \ar[d]^{\left(
             \begin{smallmatrix}
             -1 & 0 & 0 \\
             0 & -1 & 0 \\
             0 & 0 & -1 \\
             \end{smallmatrix}
           \right)} & X_6\oplus Y_5\oplus Z_4 \ar[r]^{\ \ \ \left(
             \begin{smallmatrix}
             f_6 & 0 & 0 \\
             -\varphi_6 & -g_5 & 0 \\
             -\phi_6 & \psi_5 & h_4\\
             \end{smallmatrix}
           \right)} \ar[d]^{\left(
             \begin{smallmatrix}
             1 & 0 & 0 \\
             0 & -1 & 0 \\
             0 & 0 & -1\\
             \end{smallmatrix}
           \right)} & \cdots \\
 X_3\ar[r]^{\left(
             \begin{smallmatrix}
             f_3 \\
             \varphi_3\\
             \end{smallmatrix}
           \right)} & X_4\oplus Y_3 \ar[r]^{\left(
             \begin{smallmatrix}
             -f_4 & 0 \\
             \varphi_4 & -g_3\\
              \phi_4 & \psi_3\\
             \end{smallmatrix}
           \right)}  & X_5\oplus Y_4\oplus Z_3 \ar[r]^{ \left(
             \begin{smallmatrix}
             -f_5 & 0 & 0 \\
             -\varphi_5 & -g_4 & 0 \\
             \phi_5 & \psi_4 & h_3\\
             \end{smallmatrix}
           \right)}  & X_6\oplus Y_5\oplus Z_4 \ar[r]^{\ \ \ \ \left(
             \begin{smallmatrix}
             -f_6 & 0 & 0 \\
             \varphi_6 & -g_5 & 0 \\
             \phi_6 & \psi_5 & h_4\\
             \end{smallmatrix}
           \right)}  & \cdots \\}$$
 $$\xymatrixcolsep{4pc}\xymatrixrowsep{4.5pc}\xymatrix{
  \cdots  \ar[r]^{\left(
             \begin{smallmatrix}
             f_{n-1} & 0 & 0 \\
             -\varphi_{n-1} & -g_{n-2} & 0 \\
             (-1)^n\phi_{n-1} & \psi_{n-2} & h_{n-3}\\
             \end{smallmatrix}
           \right)}  & X_n\oplus Y_{n-1}\oplus Z_{n-2} \ar[r]^{\left(
             \begin{smallmatrix}
             -\varphi_{n} & -g_{n-1} & 0 \\
             (-1)^{n+1}\phi_{n} & \psi_{n-1} & h_{n-2}\\
             \end{smallmatrix}
           \right)}\ar[d]^{\left(
             \begin{smallmatrix}
             (-1)^n & 0 & 0\\
             0 & -1 & 0\\
             0 & 0 & -1\\
             \end{smallmatrix}
           \right)} & Y_n\oplus Z_{n-1} \ar[r]^{\ \ \left(
             \begin{smallmatrix}
             \psi_n & h_{n-1} \\
             \end{smallmatrix}
           \right)}\ar[d]^{\left(
             \begin{smallmatrix}
             -1 & 0\\
             0 & -1\\
             \end{smallmatrix}
           \right)} & Z_n \ar[r]^{-\Sigma f_2\cdot h_n}\ar[d]^{-1} & \Sigma X_3\ar@{=}[d]\\
   \cdots  \ar[r]^{\left(
             \begin{smallmatrix}
             -f_{n-1} & 0 & 0 \\
             (-1)^{n-1}\varphi_{n-1} & -g_{n-2} & 0 \\
             \phi_{n-1} & \psi_{n-2} & h_{n-3}\\
             \end{smallmatrix}
           \right)\ \ \ \ }  & X_n\oplus Y_{n-1}\oplus Z_{n-2} \ar[r]^{\left(
             \begin{smallmatrix}
             (-1)^n\varphi_{n} & -g_{n-1} & 0 \\
             \phi_{n} & \psi_{n-1} & h_{n-2}\\
             \end{smallmatrix}
           \right)} & Y_n\oplus Z_{n-1} \ar[r]^{\ \ \left(
             \begin{smallmatrix}
             \psi_n & h_{n-1} \\
             \end{smallmatrix}
           \right)} & Z_n \ar[r]^{\Sigma f_2\cdot h_n} & \Sigma X_3\\
   }$$
 where the second row is the $n$-angle (2.1) given by higher octahedral axiom (N4$^\ast$). Indeed, (N4-4) is nothing but the proof of (N4) implying (N4$^*$). Moreover, by (N4-4) we can obtain the morphisms of $n$-angles hidden in (N4$^\ast$).
\end{rem}

\end{document}